\documentclass{amsart}
\usepackage[T1]{fontenc}
\usepackage{textcomp}
\usepackage[utf8x]{inputenc}
\usepackage[swedish,english]{babel}
\usepackage{ucs}
\usepackage{graphicx}
\usepackage{mathrsfs}
\usepackage{rotating} 
\usepackage{tikz}
\usepackage{tikz-cd}
\usepackage{amssymb}
\usepackage[titletoc,toc]{appendix}
\numberwithin{equation}{section} 

\DeclareFontFamily{U}{BOONDOX-calo}{\skewchar\font=45 }
\DeclareFontShape{U}{BOONDOX-calo}{m}{n}{
  <-> s*[1.05] BOONDOX-r-calo}{}
\DeclareFontShape{U}{BOONDOX-calo}{b}{n}{
  <-> s*[1.05] BOONDOX-b-calo}{}
\DeclareMathAlphabet{\mathcalboondox}{U}{BOONDOX-calo}{m}{n}
\SetMathAlphabet{\mathcalboondox}{bold}{U}{BOONDOX-calo}{b}{n}
\DeclareMathAlphabet{\mathbcalboondox}{U}{BOONDOX-calo}{b}{n}
\newtheorem{thm}{Theorem}[section]
\newtheorem{example}{Example}[section]
\newtheorem{cor}[thm]{Corollary}
\newtheorem{lem}[thm]{Lemma}
\newtheorem{prop}[thm]{Proposition}

\theoremstyle{definition}
\newtheorem{definition}[thm]{Definition}
\theoremstyle{remark}
\newtheorem{rem}[thm]{Remark}
\numberwithin{equation}{section}

\makeindex
\begin{document}

\title[{Relations in the Tautological Ring of the Universal Curve}]
{Relations in the Tautological Ring of the Universal Curve}%
\author{Olof Bergvall, olof.bergvall@hig.se}%

\begin{abstract}
 We bound the dimensions of the graded pieces of the tautological ring of the universal
 curve from below for genus up to 27 and from above for genus up to 9.
 As a consequence we obtain the precise structure of the tautological ring of the universal
 curve for genus up to 9.
 In particular, we see that it is Gorenstein for these genera.
 
 \vspace{10pt}
 \noindent MSC-classification: 14H10
\end{abstract}

\maketitle


\section{Introduction}

Chow rings of moduli spaces of curves are very large. However, Mumford \cite{mumford}
observed that one does not need the entire Chow ring in order to make interesting
intersection theoretic computations and solve many enumerative questions - the
tautological subring is enough. Very roughly speaking, this is the subring
generated by the most geometrically interesting classes of the moduli space.

In this note we investigate the tautological ring of the universal curve $\mathcal{C}_g=\mathcal{M}_{g,1}$
by combining an extension of a method of Faber \cite{faber} with results
of Liu and Xu \cite{liuxu2}.
In this way we are able to determine the structure of the tautological
ring of $\mathcal{C}_g$ up to genus $9$ and we determine its Gorenstein quotients
up to genus $27$. In particular, we show that the tautological ring of $\mathcal{C}_g$
is Gorenstein for $2 \leq g \leq 9$, see Theorem~\ref{gorthm}.

The research presented in this note was carried out at KTH in the spring of 2011
but has up to now only been presented in the somewhat obscure form \cite{bergvallmaster}.
Nevertheless, the results have gained some attention, see \cite{morita2}, \cite{yin2}, \cite{yinthesis} and \cite{yin},
and it therefore seems as though they should be presented in a way which is more accessible and
easy to read. We also remark that the methods presented here are directly applicable to
higher fiber powers of $\mathcal{C}_g$ over $\mathcal{M}_g$ and that
similar ideas could plausibly be applied to other moduli spaces of interest.

The paper is structured as follows. In Section~\ref{backgroundsec} we give
the basic definitions and present some of the known results around tautological
rings of moduli spaces of curves. In particular, we sketch a method
for producing tautological relations due to Faber \cite{faber} in Section~\ref{fabersec}.
In Section~\ref{unisec} we make an analogous construction for the tautological ring 
of the universal curve and we also present a result of Liu and Xu \cite{liuxu2} in Section~\ref{computingpgi}
which will be very important. By combining these results we are able to bound
the dimensions of graded pieces of the tautological rings from below for $g \leq 27$
and from above for $g \leq 9$. The precise results are given in Section~\ref{rankof}
and Section~\ref{generating}.

\subsection*{Acknowledgements}
The author thanks Carel Faber who supervised the project and also pointed
out the relation \ref{bequation}. I am also grateful for the many valuable 
comments and suggestions from the anonymous referees.

\pagebreak
\section{Background}
\label{backgroundsec}
Let $k$ be an algebraically closed field and let $g \geq 2$ be an integer.
We let $\mathcal{M}_{g,n}$ denote the moduli space of curves of genus $g$ 
with $n$ points, i.e. of tuples $(C, p_1, \ldots, p_n)$, where $C$ is a smooth curve of
genus $g$ over $k$ and $p_1, \ldots,p_n$ are distinct points
on $C$.
The moduli space $\mathcal{M}_{g,1}$,
is given the symbol $\mathcal{C}_g$ and we denote the
morphism $\mathcal{C}_g \to \mathcal{M}_g$
forgetting the marked point by $\pi$.
The space $\mathcal{C}_g$ is a universal curve over the dense
open subset of $\mathcal{M}_g$ parameterizing curves without automorphisms.
By abuse of terminology, $\mathcal{C}_g$ is therefore sometimes called
the universal curve over $\mathcal{M}_g$.

We denote the $n$-fold fiber product of $\mathcal{C}_g$ over
$\mathcal{M}_g$ by $\mathcal{C}_g^n$.
The space $\mathcal{C}_g^n$ parametrizes smooth curves marked with $n$,
not necessarily distinct, points. For notational convenience we
shall sometimes write $\mathcal{C}_g^1$ to mean $\mathcal{C}_g$ and
$\mathcal{C}_g^0$ to mean $\mathcal{M}_g$.
For $m \geq n$ we have various morphisms $\mathcal{C}_g^{m} \to \mathcal{C}^n_g$
forgetting $m-n$ points. Especially important are the morphisms
\begin{equation*}
 \pi_{n,i} : \mathcal{C}_g^n \to \mathcal{C}_g^{n-1},
\end{equation*}
defined by forgetting the $i$'th point. For $n=1$ we have $\pi=\pi_{1,1}$.

\subsection{Tautological rings}
The spaces $\mathcal{C}_g^n$ have Chow rings $A^{\bullet}(\mathcal{C}_g^n)$ (with rational coefficients).
These rings are however believed to be very big so, following Mumford \cite{mumford}, one instead
chooses to concentrate on a subring generated by the most important
cycles, i.e. the tautological ring. Faber and Pandharipande \cite{faberpandharipande}
has given a very natural and general definition of this system of rings, here
we choose to use Mumford's original definition.

Consider the morphism
\begin{equation*}
\pi : \mathcal{C}_g \to \mathcal{M}_g.
\end{equation*}
Let $\omega_{\pi}$ denote the relative dualizing sheaf, i.e. the sheaf of rational sections of
$\mathrm{Coker}(d \pi: \pi^*\Omega_{\mathcal{M}_g} \to \Omega_{\mathcal{C}_g})$.
Define $K$ to be the first Chern class of $\omega_{\pi}$, i.e.
\begin{equation*}
K= c_1(\omega_{\pi}) \in A^1(\mathcal{C}_g).
\end{equation*}
We use $K$ to define the so-called $\kappa$-classes
\begin{equation*}
\kappa_i = \pi_{*}(K^{i+1}) \in A^i(\mathcal{M}_g).
\end{equation*}
In particular we have $\kappa_{-1}=0$ and $\kappa_0 = 2g-2$.
We may also consider the Hodge bundle
\begin{equation*}
\mathbb{E}= \pi_{*}(\omega_{\pi}).
\end{equation*}
It is a vector bundle of rank $g$ on $\mathcal{M}_g$ whose fiber at $[C] \in \mathcal{M}_g$
is the space of holomorphic differentials on $C$. We
define the $\lambda$-classes as
\begin{equation*}
\lambda_i := c_i \left(\mathbb{E}\right) \in A^i(\mathcal{M}_g).
\end{equation*}
In particular, $\lambda_0 =1$ and $\lambda_i = 0$ if $i > g$.
The $\kappa$- and $\lambda$-classes generate a $\mathbb{Q}$-subalgebra $R^{\bullet}(\mathcal{M}_g)$
of $A^{\bullet}(\mathcal{M}_g)$ called the tautological ring.
By analogy, we introduce the relative dualizing sheaves $\omega_{\pi_{n,i}}$ of 
$\pi_{n,i}: \mathcal{C}_g^n \to \mathcal{C}_g^{n-1}$, the classes
\begin{equation*}
K_i:=c_1\left( \omega_{\pi_{n,i}} \right) \in A^1(\mathcal{C}_g^n)
\end{equation*}
and we also introduce the diagonal classes $D_{i,j}$ consisting of points
\begin{equation*}
[(C,p_1, \ldots, p_n)] \in \mathcal{C}_g^n,
\end{equation*}
such that $p_i = p_j$, $i \neq j$.

By abuse of notation we
shall also denote the pullback of $\kappa_i$ and $\lambda_i$ in $A^{\bullet}(\mathcal{C}_g^n)$
by $\kappa_i$ and $\lambda_i$, respectively. We now define
the tautological ring $R^{\bullet}(\mathcal{C}_g^n)$ of $\mathcal{C}_g^n$
as the subalgebra of $A^{\bullet}(\mathcal{C}_g^n)$ generated by the $K_i$-, $D_{i,j}$-, $\kappa$-
and $\lambda$-classes.

\subsection{Facts}

An early result concerning the tautological ring is
the following theorem of Mumford, \cite{mumford}.

\begin{thm}[Mumford]
\label{mumfordthm1} 
The classes $\lambda_i$ and $\kappa_i$ are polynomials in the classes
$\kappa_1, \ldots, \kappa_{g-2}$.
\end{thm}

For instance, we have the following relation between the $\lambda_i$
and the $\kappa_j$
\begin{equation*}
\sum_{i=0}^{\infty} \lambda_it^i = \exp \left( \sum_{i=1}^{\infty} \frac{B_{2i}\kappa_{2i-1}}{2i(2i-1)}t^{2i-1} \right),
\end{equation*}
where the $B_{2i}$ are the signed Bernoulli numbers.

As conjectured by Faber \cite{faber} and proven by Ionel \cite{ionel},
Mumford's result can be improved quite a bit.
In cohomology, the result was first obtained by Morita \cite{morita}.

\begin{thm}[Ionel \cite{ionel}]
\label{ionel1}
The $\lfloor g/3 \rfloor$ classes $\kappa_1, \ldots, \kappa_{\lfloor g/3 \rfloor}$ generate
$R^{\bullet}(\mathcal{M}_g)$,
where $\lfloor x\rfloor$ denotes the floor function of $x$.
\end{thm}

By combining the Madsen-Weiss theorem \cite{madsenweiss} and a stability result
of Boldsen \cite{boldsen} (improving results of Harer \cite{harer}) one obtains the
following.

\begin{thm}[Madsen-Weiss \cite{madsenweiss}, Boldsen \cite{boldsen}]
\label{bmwthm}
 There are no relations in $R^i(\mathcal{C}_g^n)$ for $i \leq g/3$. 
\end{thm}

\begin{rem}
 Even though Boldsen only claims the above result for $i < g/3$, the remaining case
 seems well known to experts, see e.g. Ionel \cite{ionel}.
\end{rem}

We thus have a very good understanding of the tautological ring in low degrees.
We now say something about what it known in high degrees. Since the dimension
of $\mathcal{C}_g^n$ is $3g-3+n$ there could, a priori, be nonzero tautological classes in degrees up
to $3g-3+n$. This is however far from the case as the following result of
Looijenga shows.

\begin{thm}[Looijenga \cite{looijenga2}]
\label{loovan}
$R^j(\mathcal{C}_g^n)=0$ if $j > g+n-2$ and $R^{g+n-2}(\mathcal{C}_g^n)$
is at most one-dimensional.
\end{thm}

Looijenga's theorem was improved a bit by Faber, \cite{faber2}.

\begin{thm}[Faber \cite{faber2}]
\label{fabernonvan}
The class $\kappa_{g-2}$ is non-zero in $R^{g-2}(\mathcal{M}_g)$.
\end{thm}

It follows that $R^{g-2}(\mathcal{M}_g)$ is one-dimensional. 
The non-vanishing of $R^{g+n-2}(\mathcal{C}_g^n)$ extends
easily to positive $n$.

\begin{cor}
\label{fabcor1}
$R^{g+n-2}(\mathcal{C}_ g^n)$ is one-dimensional.
\end{cor}

In the case of $\mathcal{M}_g$, Faber \cite{faber} also conjectured 
explicit proportionalities in degree $g-2$. This conjecture was proven,
first by Liu and Xu \cite{liuxu} and later by Buryak and Shadrin \cite{buryakshadrin}.
We also mention that a proof, conditional on the Virasoro conjecture for $\mathbb{P}^2$,
had previously been given by Getzler and Pandharipande \cite{getzlerpandharipande}.
A proof of the Virasoro conjecture for $\mathbb{P}^n$ was in turn announced by Givental, 
see \cite{givental1} and \cite{givental2}, although the details never seem to have appeared.
By now, Teleman \cite{teleman} has given a proof of the Virasoro conjecture for manifolds
with semi-simple quantum cohomology.

To state the result we need some notation. Let $\overline{d}=(d_1,\ldots,d_k)$ be a partition of $g-2$ into positive integers.
Let $\sigma \in S_k$ and let $\sigma=\alpha_1 \cdots \alpha_{\nu(\sigma)}$ be a decomposition of $\sigma$
into disjoint cycles. For a cycle $\alpha$ we write $|\alpha(\overline{d})|$ to denote the sum
\begin{equation*}
 |\alpha(\overline{d})| = \sum_{i \in \alpha} d_i
\end{equation*}
and we write $\kappa_{\sigma}(\overline{d})$ to denote the product
\begin{equation*}
 \kappa_{\sigma}(\overline{d}) = \prod_{i=1}^{\nu(\sigma)} \kappa_{|\alpha_i(\overline{d})|}.
\end{equation*}

\begin{thm}[Liu and Xu \cite{liuxu}]
\label{liuxuthm}
 Let $\overline{d}=(d_1,\ldots,d_k)$ be a partition of $g-2$ into positive integers. Then
\begin{equation*}
\sum_{\sigma \in \mathfrak{S}_n} \kappa_{\sigma}(\overline{d}) =
\frac{(2g-3+k)!(2g-1)!!}{(2g-1)!\prod_{j=1}^k(2d_j+1)!!} \kappa_{g-2}.
\end{equation*}
\end{thm}

Together, the results \ref{ionel1}, \ref{fabcor1} and \ref{liuxuthm} prove
two thirds of the Faber conjectures \cite{faber}. The remaining third, which asserts that the
pairing
\begin{equation}
\label{pairingeq}
R^i(\mathcal{M}_g) \times R^{g-2-i}(\mathcal{M}_g) \to R^{g-2}(\mathcal{M}_g)
\end{equation}
is perfect, remains open but it is reasonable to view the results of Petersen and Tommasi \cite{petersentommasi}
and of Petersen \cite{petersen} as evidence against the conjecture.

The following easy identities, formulated in this form by 
Harris and Mumford \cite{harrismumford}, will be fundamental for our computations.

\begin{lem}[Harris and Mumford \cite{harrismumford}]
\label{hm1}
The following identities hold in $R(\mathcal{C}_g^n)$:
\begin{align*}
& D_{i,n}D_{j,n} = D_{i,j}D_{i,n}, & i < j < n, \tag{1} \\
& D_{i,n}^2 = -K_{i}D_{i,n}, & i < n, \tag{2} \\
& K_nD_{i,n} = K_iD_{i,n}, & i < n. \tag{3}
\end{align*}
Using the above identities repeatedly, every monomial in the classes $K_i$ and $D_{ij}$
($i < j < n$) in $R(\mathcal{C}_g^n)$ can be rewritten as a monomial pulled 
back from $R(\mathcal{C}_g^{n-1})$ times either a single diagonal $D_{i,n}$ or
a power of $K_n$.

If $M$ is a monomial in $R(\mathcal{C}_g^n)$ which is pulled back
from $R(\mathcal{C}_g^{n-1})$, then
\begin{align*}
& \pi_{n*}(M \cdot D_{i,n}) = M, \tag{4} \\
& \pi_{n*}(M \cdot K_n^k)= M \cdot \pi_n^*(\kappa_{k-1}) = M \cdot \kappa_{k-1}. \tag{5}
\end{align*}
\end{lem}

\subsection{Faber's method}
\label{fabersec}

We shall now describe a method due to Faber \cite{faber} for producing relations
in the tautological ring of $\mathcal{M}_g$.

Consider the morphism $\pi_{n+1}:\mathcal{C}_g^{n+1} \to \mathcal{C}_g^n$
that forgets the $(n+1)$'st point and let $\Delta_{n+1}$ denote the sum
\begin{equation*}
\Delta_{n+1}=D_{1,n+1} + D_{2,n+1} + \cdots + D_{n,n+1}
\end{equation*}
Let $\omega_i$ be the line bundle on $\mathcal{C}_g^n$
obtained by pulling back $\omega_{\pi}$ along the projection 
$\pi_i: \mathcal{C}_g^{n} \to \mathcal{C}_g$ onto the $i$'th factor and 
define a coherent sheaf $\mathbb{F}_n$ on $\mathcal{C}_g^n$ by
\begin{equation*}
\mathbb{F}_n=\pi_{n+1*} \left(\mathcal{O}_{\Delta_{n+1}} \otimes \omega_{n+1}\right).
\end{equation*}
The sheaf $\mathbb{F}_n$ is locally free of rank $n$. 

\begin{thm}[Faber]
\label{faberthm2}
If $n \geq 2g-1$ and $j \geq n-g+1$, then $c_j(\mathbb{F}_n-\mathbb{E})=0$.
\end{thm}

Thus, if $P \in R^{\bullet}(\mathcal{C}_g^n)$ is any element, then
$P \cdot c_j(\mathbb{F}_n-\mathbb{E}) = 0$ as long as $n \geq 2g-1$ and
$j \geq n-g+1$. Pushing this down to $\mathcal{M}_g$ gives a relation
in $R^{\bullet}(\mathcal{M}_g)$. This can be done by means of Lemma~\ref{hm1}
as soon as we understand $c_j(\mathbb{F}_n-\mathbb{E})$ in terms of tautological classes.
Faber proves that
\begin{equation*}
c(\mathbb{F}_d)=(1+K_1)(1+K_2-\Delta_2)(1+K_3-\Delta_3)\cdots(1+K_d-\Delta_d)
\end{equation*}
which together with Mumford's identity \cite{mumford}
\begin{equation*}
c(\mathbb{E})^{-1} = c(\mathbb{E}^{\vee}) = \sum_{i=0}^g (-1)^i \lambda_i =
1 - \lambda_1 + \lambda_2 - \lambda_3 + \cdots + (-1)^g \lambda_g
\end{equation*}
gives an expression of the desired form.
In this way one can compute a number of relations
in $R^i(\mathcal{M}_g)$ and thus obtain an upper bound of its dimension.

We shall now discuss how to obtain a lower bound for the dimension.
Recall that $\mathrm{deg}(\kappa_i)=i$ so
if $\kappa_I=\kappa_{i_1}^{n_1} \cdot \kappa_{i_2}^{n_2} \cdots \kappa_{i_r}^{n_r}$,
then
\begin{equation*}
 \deg (\kappa_I) = \sum_{j=1}^r n_j i_j.
\end{equation*}
Let $\kappa_I$ be a monomial in the $\kappa$-classes of degree $i$
and let $\kappa_J$ be a monomial in the $\kappa$-classes of degree
$j=g-2-i$. Then $\kappa_I \cdot \kappa_J$ is a monomial
of degree $g-2$. Since $R^{g-2}(\mathcal{M}_g)$ is one-dimensional and generated by $\kappa_{g-2}$,
we may express $\kappa_I \cdot \kappa_J$ as a rational 
multiple of $\kappa_{g-2}$, $\kappa_I \cdot \kappa_J = r \cdot \kappa_{g-2}$.
We therefore make the following definition.

\begin{definition}
\label{rdef}
 Let $\kappa_I$ be a monomial of degree $g-2$ in the $\kappa$-classes.
Define $r(\kappa_I)$ to be the rational number which satisfies
\begin{equation*}
\kappa_I = r(\kappa_I) \cdot \kappa_{g-2}.
\end{equation*}
We remark that Theorem \ref{liuxuthm} may be used to calculate
the numbers $r(\kappa_I)$.
\end{definition}

From this point on we fix a monomial ordering $<_{\kappa}$ of the monomials
in the $\kappa$-classes. Which one is of no importance so the
reader may think of his or her favourite.

Recall that the partition function, $p$, is the function which
for each nonnegative integer gives the number of ways of writing it
as an unordered sum of positive integers. For instance, $p(1)=1$,
$p(2)=2$, $p(3)=3$ and $p(4)=5$. Since it is not completely uncommon
to define the partition function only for positive integers, we point
out that $p(0)=1$ (the empty partition).

\begin{definition}
\label{pgi}
 Let $i \leq g-2$ be a non-negative integer. We define the
$p(i) \times p(g-2-i)$-matrix $P_{g,i}$ as follows. Let $\kappa_k$
be the $k$th monic monomial of degree $i$ and let $\kappa_l$ be the
$l$th monic monomial of degree $g-2-i$ (according to $<_{\kappa}$).
Then the $(k,l)$th entry of $P_{g,i}$ is $r(\kappa_{k} \cdot \kappa_{l})$.
We shall refer to matrices of this type as pairing matrices.
\end{definition}

The monomials $\kappa_I$, where $I$ is a multi-index such that
\begin{equation*}
 \sum_{i_r \in I} r \cdot i_r = i,
\end{equation*}
generate $R^i(\mathcal{M}_g)$ by Theorem \ref{mumfordthm1}.
Note that every $\mathbb{Q}$-linear relation among the monomials $\kappa_I$
of degree $i$ clearly gives a linear relation among the rows
of $P_{g,i}$. Hence, if the rank of $P_{g,i}$ is $n$, then $R^i(\mathcal{M}_g)$
has dimension at least $n$.

Faber's two-step program to compute $R^{\bullet}(\mathcal{M}_g)$ for specific values of $g$
is now the following.
First, compute the rank of $P_{g,i}$ to obtain a lower bound $n$
for the dimension of $R^i(\mathcal{M}_g)$. 
Then multiply the relation
in Theorem \ref{faberthm2} by monomials and use Lemma~\ref{hm1} to push
these relations to $R^{\bullet}(\mathcal{M}_g)$. We then pick out the degree $i$
part of the relation which must be a relation in $R^i(\mathcal{M}_g)$.
By producing such relations one obtains an ideal $I \subset \mathbb{Q}(\kappa_1, \ldots, \kappa_{g-2})$
and consequently an upper bound $m$ 
for the dimension of $R^i(\mathcal{M}_g)$ as the dimension of the degree $i$ part of the quotient $\mathbb{Q}(\kappa_1, \ldots, \kappa_{g-2})/I$.
If $m=n$ one may conclude that $R^i(\mathcal{M}_g) = \mathbb{Q}(\kappa_1, \ldots, \kappa_{g-2})/I$.
Faber \cite{faber} used this idea to 
compute $R^{\bullet}(\mathcal{M}_g)$ for small values of $g$.
It was this data that lead him to state his conjectures and it also lead
to the formulation of the Faber-Zagier relations, later generalized by Pixton \cite{pixton} 
and proven in this more general form by Pandharipande, Pixton and Zvonkine \cite{pandharipandepixtonzvonkine} 
in cohomology and by Janda, see \cite{janda} and \cite{janda2},  in Chow.
Today, we know that the Faber-Zagier relations are all relations for $g \leq 23$. For $g=24$
there is one ``missing'' relation in degree $11$, i.e. there is a difference of $1$ between the upper
bound and the lower bound in this degree.

\section{The universal curve}
\label{unisec}

The aim of this project was to adapt the technique of Faber, described
in the previous section, to $R^{\bullet}(\mathcal{C}_g)$. 
To do so we first note that we may stop pushing down at 
$R^{\bullet}(\mathcal{C}_g)$ instead of at $R^{\bullet}(\mathcal{M}_g)$. 
Thus, the method for generating relations extends to $R^{\bullet}(\mathcal{C}_g)$ 
without any trouble and we thus have a way to produce upper bounds for the dimension
of $R^i(\mathcal{C}_g)$. We thus turn to the problem of finding lower bounds.

\subsection{Pairing Matrices}
\label{pairing}

Recall the matrices $P_{g,i}$, introduced in Definition~\ref{pgi}.
There are corresponding matrices related to the product structure of
$R^{\bullet}(\mathcal{C}_g)$.
To define these matrices we need a bit of preparation. 

In $R^{\bullet}(\mathcal{C}_g)$ we only have one more generator than in $R^{\bullet}(\mathcal{M}_g)$, 
namely the class $K$. Hence, Theorem \ref{mumfordthm1} gives that
$R^{\bullet}(\mathcal{C}_g)$ is generated by the monomials in $\kappa_1, \ldots, \kappa_{g-2}$
and $K$. The class $K$ has degree $1$ so a monomial $M= K^j\kappa_1^{n_1} \cdots \kappa_{g-2}^{n_{g-2}}$
has degree
\begin{equation*}
 \deg (M) = j + \sum_{i=1}^{g-2} n_i \cdot i.
\end{equation*}
We extend the monomial ordering $<_{\kappa}$ on $R^{\bullet}(\mathcal{M}_g)$
to a monomial ordering $<_{*}$ on $R^{\bullet}(\mathcal{C}_g)$ as follows.

\begin{definition}
 Let $M=K^r \kappa_I$ and $N=K^s \kappa_J$ be monomials in the $\kappa$-classes
and $K$ of the same degree. We define a monomial ordering $<_{*}$ by
\begin{align*}
 & \text{(a) } M <_* N \text{ if } r < s \text{ or,} \\
 & \text{(b) } M <_* N \text{ if } r=s \text{ and } \kappa_I <_{\kappa} \kappa_J.
\end{align*}
\end{definition}

By Theorem \ref{loovan} and Corollary \ref{fabcor1} any
monomial $M$ of degree $g-1$ is a rational multiple of $K^{g-1}$, i.e.
$M = s(M) \cdot K^{g-1}$ for some rational number $s(M)$. 
By Lemma \ref{hm1} we have that 
$\pi_{*}(M) = s(M) \cdot \kappa_{g-2}$. 

\begin{definition}
 Let $M$ be the $k$'th monic monomial of degree $i$ according to $<_*$ and
let $N$ be the $l$th monic monomial of degree $g-1-i$ according to $<_*$.
Define $s^i_{k,l}$ as the rational number satisfying $\pi_{*}(M \cdot N) = s^i_{k,l} \kappa_{g-2}$ and let
\begin{equation*}
 Q_{g,i} = (s^i_{k,l}).
\end{equation*}
\end{definition}

The dimensions of $Q_{g,i}$ are
\begin{equation*}
 \left( \sum_{r=0}^i p(r) \right) \times \left( \sum_{r=0}^{g-1-i} p(r) \right),
\end{equation*}
where $p$ is the partition function.
Just as for the matrices $P_{g,i}$, the rank of $Q_{g,i}$ determines a lower
bound for the dimension of $R^i(\mathcal{C}_g)$.

To explain the relationship between the matrices $Q_{g,i}$ and $P_{g,i}$
in more detail it is convenient to introduce some notation.

\begin{definition}
Let $j \leq i$ be a positive integer. Define $P_{g,i}^j$ as the $p(i-j) \times p(g-2-i)$-submatrix of 
$P_{g,i}$ consisting of the rows of $P_{g,i}$ which are labeled by monomials $\kappa_I$
containing at least one factor $\kappa_j$. We also define
\begin{equation*}
P_{g,i}^0 = (2g-2)\cdot P_{g,i}
\end{equation*}
and
\begin{equation*}
P_{g,i}^{-1} = \text{ the zero matrix of size } p(i+1) \times p(g-2-i).
\end{equation*}
\end{definition} 

We are now ready to state the following Proposition.

\begin{prop}
\label{matrixprop1}
(a) Let $Q_{g,i}$ and $P_{g,j}^r$ be defined as above and let $i \geq 1$. Then,
\begin{equation*}
Q_{g,i} = \left( \begin{array} {cccccc}
P_{g,i-1}^{-1} & P_{g,i}^0 & P_{g,i+1}^1 & P_{g,i+2}^2 & \cdots & P_{g,g-2}^{g-2-i} \\
P_{g,i-1}^0 & P_{g,i}^1 & P_{g,i+1}^2 & \cdots & \cdots & \vdots \\
P_{g,i-1}^1 & P_{g,i}^2 & \ddots & & & \vdots \\
P_{g,i-1}^2 & \vdots & & \ddots & & \vdots \\
\vdots & \vdots & & & \ddots & \vdots \\
P_{g,i-1}^{i-1} & \cdots & \cdots & \cdots & \cdots & P_{g,g-2}^{g-2}
\end{array} \right).
\end{equation*} 

(b) The rank of $Q_{g,0}$ is $1$.
\end{prop}

\begin{proof}
(a) Denote the monomial labeling the $r$th row of $Q_{g,i}$ by $N_r$ and the monomial
labeling the $s$th column of $Q_{g,i}$ by $N_s$.
Consider first the submatrix of $Q_{g,i}$ corresponding to rows and columns
labeled by monomials $N_r$ and $N_s$ not containing a factor $K$. Then 
$N_r \cdot N_s$ projects to $0$ so this submatrix consists entirely of
zeros. With the above notation, this submatrix is equal to $P_{g,i-1}^{-1}$.

Now consider a submatrix $C$ of $Q_{g,i}$ corresponding to rows and columns
labeled by monomials $N_r$ and $N_s$ such that ,
\begin{itemize}
\item[\emph{(i)}] $N_r=K^{n_r}N'_r$ and $N_s=K^{n_s}N'_s$ where $K$ does not divide $N'_r$ or $N'_s$ and,
\item[\emph{(ii)}] not both $n_r$ and $n_s$ are zero.
\end{itemize}
Then,
\begin{equation*}
\pi_{*}(N_r \cdot N_s) = \pi_{*}(K^{n_r+n_s} \cdot N'_r \cdot N'_s) = \kappa_{n_r+n_s-1} \cdot N'_r \cdot N'_s.
\end{equation*}
Note that $\kappa_{n_r+n_s-1} \cdot N'_r$ is a monomial in the $\kappa_i$'s of degree
$i+n_s-1$ containing a factor $\kappa_{n_r+n_s-1}$ and that $N'_s$ is a monomial
of degree $g-1-i-n_s=g-2-(i+n_s-1)$ in the $\kappa_i$'s. Further, every monomial
in the $\kappa_i$'s of degree $i+n_s-1$ containing a factor $\kappa_{n_r+n_s-1}$
is the image of some monomial $K^{n_r+n_s} \cdot N'_r$ and every
polynomial of degree $g-2-(i+n_s-1)$ is represented by the $N'_s$'s. By our choice
of monomial order labeling the rows and columns of $Q_{g,i}$ we now see that 
$C=P_{g,i+n_s-1}^{n_r+n_s-1}$. This completes the proof of (a).

(b) The only row of $Q_{g,0}$ is labeled by $1$. The last column of $Q_{g,0}$
is labeled by $K^{g-1}$. Hence, the last entry of 
$Q_{g,0}$ is $1$ and we conclude that the rank is one.
\end{proof}

The merit of Proposition~\ref{matrixprop1} is that it tells us how to compute
the matrices $Q_{g,i}$ without having to project monomials of $R^{\bullet}(\mathcal{C}_g)$
down to $R^{\bullet}(\mathcal{M}_g)$. Hence, we have reduced the problem of computing the matrices
$Q_{g,i}$ to computing the matrices $P_{g,i}$, which are smaller and easier to compute. 
We shall describe a rather efficient way of doing this shortly. However,
we first make a few observations which reduce the calculations a bit.

Firstly, let $M$ be the $k$th monic monomial of degree $i$ and let $N$ be the
$l$th monic monomial of degree $g-1-i$. 
By definition we have $\pi_{*}(M \cdot N)=s^i_{k,l}\cdot \kappa_{g-2}$ and 
$\pi_{*}(N \cdot M) = s^{g-1-i}_{l,k}\cdot \kappa_{g-2}$.
We thus see that $Q_{g,g-1-i}=Q_{g,i}^T$. Similarly, we have $P_{g,g-2-i}=P_{g,i}^T$.
Hence, we only have to compute $P_{g,i}$ for $i \leq \lceil (g-2)/2 \rceil$
and we only have to compute the rank of $Q_{g,i}$ for $i \leq \lceil (g-1)/2\rceil$.

Secondly, Theorem~\ref{bmwthm} states that there are no relations in
degrees less than $g/3$. In other words, the matrices $Q_{g,i}$ have full
rank for $i < \lfloor g/3\rfloor$.
Thus, what needs to be computed is the rank
of $Q_{g,i}$ for $\lfloor g/3 \rfloor < i \leq \lceil g/2 \rceil$. This is done
by means of Proposition \ref{matrixprop1} and the following
algorithm of Liu and Xu \cite{liuxu2}.

\subsection{Computing pairing matrices}
\label{computingpgi}

In this section we describe an algorithm due to Liu and Xu \cite{liuxu2}
by means of which one may efficiently compute the matrices $P_{g,i}$.

Let $\textbf{m} = (m_1, m_1, \cdots)$ be a sequence of non-negative
integers with only finitely many of the $m_i$ nonzero. 
The set of such sequences is a monoid under coordinatewise addition.
Define
\begin{equation*}
|\textbf{m}| = \sum_{i=1}^{\infty} i \cdot m_i, \quad
||\textbf{m}|| = \sum_{i=1}^{\infty} m_i, \quad
\textbf{m}! = \prod_{i=1}^{\infty} m_i!.
\end{equation*}
A sequence $\textbf{m}$ determines a monomial $\kappa_{\textbf{m}}$ as
\begin{equation*}
\kappa_{\textbf{m}} = \prod_{m_i \in \textbf{m}} \kappa_i^{m_i}
\end{equation*}
Inductively define constants $\beta_{\textbf{m}}$ by setting $\beta_{\textbf{0}}=1$ and requiring
\begin{equation*}
\sum_{\textbf{m}' + \textbf{m}''=\textbf{m}} \frac{(-1)^{||\textbf{m}'||}\beta_{\textbf{m}'}}{\textbf{m}''!(2|\textbf{m}''|+1)!!} = 0 \quad \text{when } \textbf{m} \neq \textbf{0}
\end{equation*}
and constants $\gamma_{\textbf{m}}$ as
\begin{equation*}
\gamma_{\textbf{m}} = \frac{(-1)^{||\textbf{m}||}}{\textbf{m}!(2|\textbf{m}|+1)!!}.
\end{equation*}
The constants $\beta_{\textbf{m}}$ and $\gamma_{\textbf{m}}$ can be used to
define new constants $C_{\textbf{m}}$
\begin{equation*}
C_{\textbf{m}} = \sum_{\textbf{m}'+\textbf{m}''=\textbf{m}} 2|\textbf{m}'|\beta_{\textbf{m}'} \gamma_{\textbf{m}''}.
\end{equation*}
Now let $|\textbf{m}| \leq g-2$ and define further constants $F_g(\textbf{m})$
via
\begin{equation*}
|\textbf{m}| \cdot F_g(\textbf{m}) = (g-1) \cdot \sum_{\substack{\textbf{m}'+\textbf{m}''= \textbf{m} \\ \textbf{m}' \neq 0}} C_{\textbf{m}'} F_g(\textbf{m}''),
\end{equation*}
and let $F_g(\textbf{0})=1$.
We can now state the following result of Liu and Xu \cite{liuxu2}.

\begin{thm}[Liu and Xu \cite{liuxu2}]
\label{liuxumatrix}
Let $|\textbf{m}| =g-2$ and let $r(\kappa_{\textbf{m}})$ be as defined
in Definition \ref{rdef}. Then $r(\kappa_{\textbf{m}})$ is given by
\begin{equation*}
r(\kappa_{\textbf{m}}) = \frac{(2g-3)!! \cdot \textbf{m}!}{2g-2} \cdot F_g(\textbf{m}).
\end{equation*}
\end{thm}

Theorem~\ref{liuxumatrix} gives a very efficient method to compute $P_{g,i}$.
The method is especially nice if one wants to compute many different $P_{g,i}$,
since much of the work can be reused, so the theorem suits our purposes very well.

Since the definitions are somewhat involved it might be helpful to see an example
in order to decipher them.

\pagebreak[2]
\begin{example}
\label{liuxuex}
Let $g=4$ and consider $r(\kappa_{(2,0,\cdots)})$.
First take $\textbf{m}=(1,0,0,\cdots)$. Then
\begin{align*}
 0 & = \frac{(-1)^{||\textbf{0}||}\beta_{\textbf{0}}}{(1,0,0, \cdots)! (2|(1,0,0, \cdots)|+1)!!} +
 \frac{(-1)^{||(1,0,0,\cdots)||}\beta_{(1,0,0, \cdots)}}{\textbf{0}!(2|\textbf{0}|+1)!!} = \\
 & = \frac{1 \cdot 1}{1\cdot (2 \cdot 1 + 1)!!} - \frac{\beta_{(1,0,0,\cdots)}}{1\cdot(2 \cdot 0+1)!!} = \\
 & = \frac{1}{3} - \beta_{(1,0,0,\cdots)}.
\end{align*}
Hence, $\beta_{(1,0,0,\cdots)}=\frac{1}{3}$. 
A similar computation for $\textbf{m}=(2,0,0, \cdots)$ gives $\beta_{(2,0,\cdots)}=\frac{7}{90}$.
We also compute $\gamma_{\textbf{0}}=1$,
$\gamma_{(1,0,\cdots)}=\frac{1}{3}$ and $\gamma_{(2,0,\cdots)} = \frac{1}{30}$.
We continue by computing the $C_{\textbf{m}}$. For instance we get
\begin{equation*}
 C_{(1,0,\cdots)} = 2 |(1,0,\cdots)|\beta_{(1,0,\cdots)}\gamma_{\textbf{0}} + 2|\textbf{0}|\beta_{\textbf{0}}\gamma_{(1,0,\cdots)}=2 \cdot 1 \cdot \frac{1}{3} \cdot 1 = \frac{2}{3}.
\end{equation*}
A similar computation gives $C_{(2,0,\cdots)} = \frac{4}{45}$.

Up to this point, the computations are valid for all $g \geq 2$. 
However, $F_g(\textbf{m})$ depends on $g$, which in our case is $4$. We get
\begin{align*}
 |(1,0,\cdots)|F_4((1,0,\cdots)) = (4-1)\cdot \frac{2}{3} \cdot 1
\end{align*}
so $F_4((1,0,\cdots))=2$. A similar computation gives that $F_4((2,0,\cdots))=\frac{32}{15}$.
Lemma \ref{liuxumatrix} now gives that
\begin{equation*}
 r(\kappa_{(2,0,\cdots)}) = \frac{(2\cdot 4-3)!! \cdot (2,0,\cdots)!}{2 \cdot 4-2} \cdot \frac{32}{15} = \frac{15 \cdot 2}{6} \cdot \frac{32}{15} = \frac{32}{3}.
\end{equation*}
Since $\kappa_{(2,0,\cdots)}=\kappa_1^2$, this is another way of expressing that
in $R^2(\mathcal{M}_4)$, the relation
\begin{equation*}
 \kappa_1^2 = \frac{32}{3} \cdot \kappa_2
\end{equation*}
holds. This relation can also be found in \cite{faber}.
\end{example}

\subsection{Ranks of pairing matrices}
\label{rankof}

Using Proposition~\ref{matrixprop1} and Theorem~\ref{liuxumatrix}
we have constructed a Maple\footnote{Maple\copyright is a trademark of Waterloo Maple Inc.} 
program for computing the rank
of $Q_{g,i}$. The results for $g \leq 27$ are shown in Table~\ref{ranktable}
below.

Write $g=3k-l-1$ with $k$ a positive integer and $l$ a non-negative integer.
In \cite{faber}, Faber remarked that the computational evidence suggests
that the dimension of the degree $k$ part of the kernel 
of the homomorphism
\begin{equation*}
 \varphi: \mathbb{Q}[x_1, \ldots, x_{g-2}] \to R^{\bullet}(\mathcal{M}_g)
\end{equation*}
sending $x_i$ to $\kappa_i$
only depends on
$l$ as long as $2k \leq g-2$. Under this assumption, $a(l)$ is defined
to be $\mathrm{dim}(\mathrm{ker}(\varphi)_{k})$. The numbers $a(l)$
has been computed in \cite{faber} for $0 \leq l \leq 9$. This has
later been extended to $l \leq 14$ in \cite{liuxu2}. We show the
results for $0 \leq l \leq 11$ in Table \ref{atable}.

Faber and Zagier have guessed that $a(l)$ equals the number of partitions 
of $l$ without any parts other than $2$ which are congruent to  $2$ modulo $3$.
The guess is supported by the following (see also \cite{faber3}). Let $\textbf{p}=\{p_1,p_3,p_4,p_6,p_7,p_8,p_9, \ldots\}$
be a collection of variables indexed by the positive integers not congruent
to $2$ modulo $3$. Define
\begin{equation*}
 \Psi(t,\textbf{p}) = \sum_{i=0}^{\infty}t^ip_{3i} \sum_{j=0}^{\infty} \frac{(6j)!}{(3j)!(2j)!}t^j +
\sum_{i=0}^{\infty}t^ip_{3i+1} \sum_{j=0}^{\infty} \frac{(6j)!}{(3j)!(2j)!}\frac{6j+1}{6j-1}t^j,
\end{equation*}
where we take $p_0 = 1$. Let $\sigma=(\alpha_1,0,\alpha_3,\alpha_4,0,\alpha_6 \ldots)$
be a sequence of non-negative integers with all coordinates with indices congruent to $2$ modulo $5$
equal to zero. Define
\begin{equation*}
 \textbf{p}^{\sigma} = p_1^{\alpha_1}p_3^{\alpha_3}p_4^{\alpha_4}\cdots.
\end{equation*}
Define constants $C_r(\sigma)$ via
\begin{equation*}
 \log \left( \Psi(t,\textbf{p}) \right) = \sum_{\sigma} \sum_{r=0}^{\infty} C_r(\sigma)t^r\textbf{p}^{\sigma}.
\end{equation*}
We use these constants to define
\begin{equation*}
 \gamma = \sum_{\sigma} \sum_{i=0}^{\infty} C_r(\sigma) \kappa_r t^r \textbf{p}^{\sigma}.
\end{equation*}
It was shown by Faber and Zagier that the relation
\begin{equation*}
 \left[ \exp \left( -\gamma \right) \right]_{t^r\textbf{p}^{\sigma}} = 0,
\end{equation*}
holds in the Gorenstein quotient of $R^{\bullet}(\mathcal{M}_g)$ when
$g-1+|\sigma|<3r$ and $g \equiv r+|\sigma|+1 \mod 2$. These are the so-called
FZ-relations. It has been shown by Pandharipande and Pixton, see \cite{pandharipandepixton} and \cite{pandharipandepixton2},
that these relations also hold in $R^{\bullet}(\mathcal{M}_g)$. These relations
are sufficiently many for codimensions $\leq \lfloor (g-2)/2 \rfloor$,
but it is not clear whether these relations are linearly independent or
not. Note the central role of positive integers not congruent to $2$ modulo
$3$ in the above - this has now been explained in terms of $3$-spin structures, 
see \cite{pandharipandepixtonzvonkine}.

With the above in mind, it might be interesting to investigate whether a
similar behaviour can be observed in $R^{\bullet}(\mathcal{C}_g)$. 
We therefore introduce the homomorphism
\begin{equation*}
 \hat{\varphi} : \mathbb{Q}[x_1, \ldots, x_{g-2},y] \to R^{\bullet}(\mathcal{C}_g)
\end{equation*}
sending $x_i$ to $\kappa_i$ and $y$ to $K$ and note that
note that the expected dimension of the degree $k$ part of $\mathrm{dim}(\mathrm{ker}(\hat{\varphi}))$
is given through the formula
\begin{equation*}
n = \sum_{i=0}^k p(i) - \text{rank}(Q_{g,k}).
\end{equation*}
Here $p(i)$ is the partition function extended with $p(0)=1$.
The computations for $l \leq 9$ suggested that the number $n$
is a function of $l$ only, as long as $2k \leq g-1$, but for $l\geq 10$ this pattern does not
persist.
Nevertheless, we shall momentarily pretend that $n$ is a function
of $l$. 
We show the computations of $n$ for $0 \leq l \leq 11$ in Table \ref{btable}.

Using Table \ref{btable}, a formula $b(l)$
for $n$ as a function of $l$ was guessed by Faber
\begin{equation}
\label{bequation}
b(l) = \sum_{\substack{i=0\\i \not\equiv 2 \text{ (mod 3)}}}^l a(l-i),
\end{equation}
where $a$ is the $a$-function discussed above. As is easily shown by induction, $b(l)$
satisfies the following recursive formula
\begin{equation*}
b(l) = 2\sum_{i=0}^{l-1}a(i) + a(l) - b(l-1)-b(l-2), \quad l \geq 2,
\end{equation*}
with initial values $b(0)=a(0)$ and $b(1)=a(0)+a(1)$.

Our guess $b(l)$, gives the right number of relations $n$ when $0 \leq l \leq 9$ but it
gives the value $b(10)=90$ instead of the value $n=91$ which
was obtained by computing the rank of $Q_{25,12}$. To investigate
the matter further I computed the rank of $Q_{28,13}$ and
$Q_{31,14}$. Both computations gave the predicted value $n=b(10)=90$
which suggests that $Q_{25,12}$ is exceptional. Noteworthy is
that the anomaly occurs in the middle degree, $(g-1)/2$.

The above results suggest that $n$ may exhibit a similar
behaviour in the middle degree also for $g > 25$. If this
is so, we expect an anomaly for $g=27$, $k=13$. The rank of $Q_{27,13}$ gives
$n=120$ while $b(11)=119$. Computing
the rank of $Q_{30,14}$ again yields the predicted value, $n=b(11)=119$.

One way to avoid this anomaly would be to require
$2k \leq g-2$ instead of $2k \leq g-1$, although
this is not very appealing (and very ad hoc). It might be interesting to
recall that the method of Faber has been unsuccessful in
proving the Faber conjectures in $R^{11}(\mathcal{M}_{24})$. Note that
also here the problem arises in the middle degree.

\begin{table}
\resizebox{0.95\textwidth}{!}{$
\begin{tabular}{ccccccccccccccccccccccccccccc}
$g$ $\setminus$ $i$ & \vline & 0 & 1 & 2 & 3 & 4 & 5 & 6 & 7 & 8 & 9 & 10 & 11 & 12 & 13 & 14 & 15 & 16 & 17 & 18 & 19 & 20 & 21 & 22 & 23 & 24 & 25 & 26 \\
\hline
2 &  \vline & 1 & 1 \\
3 &  \vline & 1 & 2 & 1 \\
4 &  \vline & 1 & 2 & 2 & 1 \\
5 &  \vline & 1 & 2 & 3 & 2 & 1 \\
6 &  \vline & 1 & 2 & 4 & 4 & 2 & 1 \\
7 &  \vline & 1 & 2 & 4 & 5 & 4 & 2 & 1 \\
8 &  \vline & 1 & 2 & 4 & 6 & 6 & 4 & 2 & 1 \\
9 &  \vline & 1 & 2 & 4 & 7 & 9 & 7 & 4 & 2 & 1 \\
10 & \vline & 1 & 2 & 4 & 7 & 10 & 10 & 7 & 4 & 2 & 1 \\
11 & \vline & 1 & 2 & 4 & 7 & 11 & 13 & 11 & 7 & 4 & 2 & 1 \\
12 & \vline & 1 & 2 & 4 & 7 & 12 & 16 & 16 & 12 & 7 & 4 & 2 & 1 \\
13 & \vline & 1 & 2 & 4 & 7 & 12 & 17 & 20 & 17 & 12 & 7 & 4 & 2 & 1 \\
14 & \vline & 1 & 2 & 4 & 7 & 12 & 18 & 24 & 24 & 18 & 12 & 7 & 4 & 2 & 1 \\
15 & \vline & 1 & 2 & 4 & 7 & 12 & 19 & 27 & 31 & 27 & 19 & 12 & 7 & 4 & 2 & 1 \\
16 & \vline & 1 & 2 & 4 & 7 & 12 & 19 & 28 & 35 & 35 & 28 & 19 & 12 & 7 & 4 & 2 & 1 \\
17 & \vline & 1 & 2 & 4 & 7 & 12 & 19 & 29 & 39 & 45 & 39 & 29 & 19 & 12 & 7 & 4 & 2 & 1 \\
18 & \vline & 1 & 2 & 4 & 7 & 12 & 19 & 30 & 42 & 53 & 53 & 42 & 30 & 19 & 12 & 7 & 4 & 2 & 1 \\
19 & \vline & 1 & 2 & 4 & 7 & 12 & 19 & 30 & 43 & 57 & 64 & 57 & 43 & 30 & 19 & 12 & 7 & 4 & 2 & 1 \\
20 & \vline & 1 & 2 & 4 & 7 & 12 & 19 & 30 & 44 & 61 & 75 & 75 & 61 & 44 & 30 & 19 & 12 & 7 & 4 & 2 & 1  \\
21 & \vline & 1 & 2 & 4 & 7 & 12 & 19 & 30 & 45 & 64 & 83 & 94 & 83 & 64 & 45 & 30 & 19 & 12 & 7 & 4 & 2 & 1 \\
22 & \vline & 1& 2& 4& 7& 12& 19& 30& 45& 65& 87& 106& 106& 87& 65& 45& 30& 19& 12& 7& 4& 2& 1 \\
23 & \vline & 1& 2& 4& 7& 12& 19& 30& 45& 66& 91& 117& 131& 117& 91& 66& 45& 30& 19& 12& 7& 4& 2& 1 \\
24 & \vline & 1& 2& 4& 7& 12& 19& 30& 45& 67& 94& 125& 150& 150& 125& 94& 67& 45& 30& 19& 12& 7& 4& 2& 1 \\
25 & \vline & 1& 2& 4& 7& 12& 19& 30& 45& 67& 95& 129& 162& 181& 162& 129& 95& 67& 45& 30& 19& 12& 7& 4& 2& 1 \\
26 & \vline & 1& 2& 4& 7& 12& 19& 30& 45& 67& 96& 133& 173& 208& 208& 173& 133& 96& 67& 45& 30& 19& 12& 7& 4& 2& 1 \\
27 & \vline & 1& 2& 4& 7& 12& 19& 30& 45& 67& 97& 136& 181& 227& 253& 227& 181& 136& 97& 67& 45& 30& 19& 12& 7& 4& 2& 1
\end{tabular}
$}
\caption{The rank of $Q_{g,i}$ for $2 \leq g \leq 27$ and $0 \leq i \leq 26$.}
\label{ranktable}
\end{table}

\renewcommand{\arraystretch}{0.3}
\renewcommand{\tabcolsep}{0.2cm}
\begin{table}
\begin{tabular}{c|cccccccccccc}
$l$  & 0 & 1 & 2 & 3 & 4 & 5 & 6 & 7 & 8 & 9 & 10 & 11 \\ \hline \\
$a(l)$ & 1 & 1 & 2 & 3 & 5 & 6 & 10 & 13 & 18 & 24 & 33 & 41 
\end{tabular}
\caption{The $a$-function for $0 \leq l \leq 11$. The values 
for $l \leq 9$ can be found in \cite{faber} while $a(10)$ and $a(11)$ are found in
\cite{liuxu2}.}
\label{atable}
\end{table}

\renewcommand{\arraystretch}{0.3}
\renewcommand{\tabcolsep}{0.2cm}
\begin{table}
\begin{tabular}{c|cccccccccc|c|cc}
$l$  & 0 & 1 & 2 & 3 & 4 & 5 & 6 & 7 & 8 & 9 & 10 & 11 \\ \hline \\
$n$ & 1 & 2 & 3 & 6 & 10 & 14 & 22 & 33 & 45 & 64  & 90 \, (91) & 119 \, (120)\\ \hline \\
\# & 8 & 7 & 6 & 6 & 5 & 4 & 4 & 3 & 2 & 2  & 2 \, (1) & 1 \, (1)
\end{tabular}
\caption{$n$ for $0 \leq l \leq 11$. \# is the number of
$g$ for which $n$ has been computed. The numbers in parentheses
are values for which the expected behaviour fails along with how many times that happened
for each $l$.}
\label{btable}
\end{table}

\subsection{Generating Relations}
\label{generating}

We earlier described a method for generating relations. Even though the method is rather easy
in principle, its computational complexity is quite an obstacle.
We shall therefore discuss a few tricks which have helped to make the computations
more efficient.

The first step of the algorithm is to pick a monomial $M$ in $R^{\bullet}(\mathcal{C}_g^{2g-1})$ 
in the $K$ and $D_{i,j}$-classes. However, the set of all such polynomials
is much too large already for low $g$. The computations so far suggest that
the algorithm described below produces enough relations.

Suppose that we want to produce relations in $R^i(\mathcal{C}_g)$ by
multiplying the relation $c_j(\mathbb{F}_{2g-1}-\mathbb{E}) \in R^j(\mathcal{C}_g^{2g-1})$ 
by a monomial $M$ and then pushing down. Since the degree drops by $2g-2$
and since $c_j(\mathbb{F}_{2g-1}-\mathbb{E})$ has degree $j$,
the degree $d$ of the monomial must be $d=i+2g-2-j$. Choose $q=2g+2i-2j+1$ and
define monomials in the following way.
\begin{itemize}
\item[(a)] $\,$ Define $M_0=D_{1,2}D_{1,3} \cdots D_{1,q}D_{q+1,q+2}D_{q+3,q+4} \cdots D_{2g-2,2g-1}$,
\item[(b)] $\,$ for $r=0,1,\ldots,q-3$, replace $D_{1,q-r}$ by $D_{q-r,q-r+1}$ in $M_r$ to obtain $M_{r+1}$.
\end{itemize}
Each $M_r$ is a monomial of degree $i+2g-2-j$ and $M_rc_j(\mathbb{F}_{2g-1}-\mathbb{E})$
will thus give a relation in $R^i(\mathcal{C}_g)$ when pushed down.

The second step is to calculate $M \cdot c_j (\mathbb{F}_{2g-1} - \mathbb{E})$
for suitable choices of $j$. As stated earlier, we have
\begin{equation*}
c(\mathbb{F}_{2g-1}) = (1+K_1)(1+K_2-\Delta_2)(1+K_3-\Delta_3) \cdots (1+K_{2g-1}-\Delta_{2g-1}),
\end{equation*}
and
\begin{equation*}
c(\mathbb{E})^{-1} = \sum_{i=0}^g (-1)^i \lambda_i.
\end{equation*}
Hence
\begin{align*}
c(\mathbb{F}_{2g-1}-\mathbb{E}) & = c_0(\mathbb{F}_{2g-1}) + c_1(\mathbb{F}_{2g-1}) -\lambda_1c_0(\mathbb{F}_{2g-1}) + c_2(\mathbb{F}_{2g-1}) - \\ 
 &-\lambda_1c_1(\mathbb{F}_{2g-1}) + \lambda_2c_0(\mathbb{F}_{2g-1}) + \cdots
\end{align*}
If we identify the degree $k$ part we obtain the formula
\begin{equation*}
c_k(\mathbb{F}_{2g-1}-\mathbb{E}) = \sum_{i=0}^{k} (-1)^i \lambda_ic_{k-i}(\mathbb{F}_{2g-1}). \tag{1}
\end{equation*}
As pointed out in \cite{faber}, we have
\begin{equation*}
c_k(\mathbb{F}_n)=c_k(\mathbb{F}_{n-1}) + (K_n-\Delta_n)c_{k-1}(\mathbb{F}_{n-1}).
\end{equation*}
No term of $c_j(\mathbb{F}_{n-1})$ has a factor $K_n$ or $D_{i,n}$.
Hence, if $P$ is a polynomial in $K_i$ and $D_{i,j}$ then,
$\pi_{n*}(P \cdot c_j(\mathbb{F}_{n-1}))=\pi_{n*}(P) \cdot c_j(\mathbb{F}_{n-1})$.
By putting the pieces together we obtain
\begin{equation*}
\pi_{n*}(M c_k(\mathbb{F}_n)) =
\pi_{n*}(M)c_k(\mathbb{F}_{n-1}) + \pi_{n,n*}(M (K_n-\Delta_n))c_{k-1}(\mathbb{F}_{n-1}). \tag{2}
\end{equation*}
Using formulas (1) and (2), the computations become more manageable.

Several Maple procedures has been written for performing these computations.
These procedures has then been used to find the necessary number of relations
for $2 \leq g \leq 9$. In other words, we have the following.

\begin{thm}
\label{gorthm}
 The tautological ring $R^{\bullet}(\mathcal{C}_g)$ is Gorenstein for $2 \leq g \leq 9$.
\end{thm}

No higher genera have been attempted since the computations
are expected to take unfeasibly long time. However, shortly after our results first appeared,
Yin \cite{yin} was able to prove that $R^{\bullet}(\mathcal{C}_g)$ is Gorenstein for $g$ up to $19$
using completely different methods. Below, we present
the relations for $g=2$, $3$ and $4$. The other relations, as well as
the Maple code, are available from the author upon request.

\subsubsection*{The case $g=2$}

Since $\kappa_{0}=2g-2=2$, there should be no relation in degree zero.
In degree one there should be one relation. Multiplying $c_2(\mathbb{F}_3-\mathbb{E})$
by $D_{2,3}$ and pushing down to $R^*(\mathcal{C}_2)$ yields the relation
$\frac{5}{3} \kappa_1 = 0$. Hence, $K \neq 0$ and $\kappa_1=0$.
This is no surprise, since $\kappa_1$ is the pullback of $\kappa_1$ in
$R^*{\mathcal{M}_g}$, which is zero by \cite{faber}. The result
also follows from Theorem~\ref{loovan} and Theorem~\ref{fabernonvan}.

\subsubsection*{The case $g=3$}

Since $g/3=1$ we should have no relations in degrees zero and one. In degree
two we should have three relations (and will have, by Theorems~\ref{loovan} and~\ref{fabernonvan}). 
Multiplying $c_3(\mathbb{F}_5-\mathbb{E})$
with $D_{1,2}D_{1,3}D_{4,5}$ respectively $D_{1,2}D_{3,4}D_{4,5}$ and
pushing down to $R^*(\mathcal{C}_3)$ yields the relations
\begin{equation*}
42K^2-\frac{21}{2}K\kappa_1+\frac{7}{48}\kappa_1^2=0, \quad 
126K^2-\frac{63}{2}K\kappa_1+\frac{41}{48}\kappa_1^2-6\kappa_2=0.
\end{equation*}
Multiplying $c_4(\mathbb{F}_{5}-\mathbb{E})$ with $D_{2,3}D_{4,5}$ and pushing
down yields the relation
\begin{equation*}
56K^2-14K\kappa_1+\frac{47}{12}\kappa_1^2-20\kappa_2=0.
\end{equation*}
These three relations are linearly independent, so we are done. If
we solve the equations we see that
\begin{equation*}
\kappa_1^2=\kappa_2=0, \quad \text{and} \quad K \kappa_1 = 4 K^2.
\end{equation*}

\subsubsection*{The case $g=4$}

We expect to find two relations in degree $2$ and six in degree $3$.
Multiplying $c_4(\mathbb{F}_7-\mathbb{E})$ with $D_{1,2}D_{1,3}D_{4,5}D_{6,7}$
respectively $D_{1,2}D_{3,4}D_{4,5}D_{6,7}$ and pushing down
yields the relations
\begin{equation*}
420K^2-70K\kappa_1+\frac{115}{6}\kappa_1^2-150\kappa_2=0, \quad  
120K^2-20K\kappa_1+\frac{10}{3}\kappa_1^2-20\kappa_2=0.
\end{equation*}
These relations are linearly independent so we are done in degree
$2$. We solve the equations to obtain
\begin{equation*}
\kappa_1^2 = \frac{32}{3} \kappa_2, \quad \text{and} \quad K\kappa_1 = 6 K^2 + \frac{7}{9} \kappa_2.
\end{equation*}
Note that the first of these relations is the relation we obtained
in $R(\mathcal{M}^2_4)$ in Example \ref{liuxuex}.

In degree $3$ we have the six linearly independent relations which can
be written as
\begin{equation*}
\kappa_3= \kappa_2 \kappa_1 = \kappa_1^3 = 0, \quad K_1^2 \kappa_1 = \frac{32}{3} K_1^3, \quad K_1 \kappa_1^2 = 64 K_1^3, \quad K_1\kappa_1 = 6 K_1^3.
\end{equation*}

\bibliographystyle{acm}

\renewcommand{\bibname}{References} 

\bibliography{references} 

\begin{thebibliography}{10}

\bibitem{bergvallmaster}
{\sc Bergvall, O.}
\newblock {\em Relations in the tautological ring of the universal curve}.
\newblock Master thesis, KTH, 2011.

\bibitem{boldsen}
{\sc Boldsen, S.}
\newblock Improved homological stability for the mapping class group with
  integral or twisted coefficients.
\newblock {\em Math. Z. 270}, 1-2 (2012), 297--329.

\bibitem{buryakshadrin}
{\sc Buryak, A., and Shadrin, S.}
\newblock A new proof of {F}aber's intersection number conjecture.
\newblock {\em Adv. Math. 228}, 1 (2011), 22--42.

\bibitem{faber2}
{\sc Faber, C.}
\newblock A non-vanishing result for the tautological ring of $\mathcal{M}_g$.
\newblock arXiv:math/9711219, 1997.

\bibitem{faber}
{\sc Faber, C.}
\newblock A conjectural description of the tautological ring of the moduli
  space of curves.
\newblock In {\em Moduli of curves and abelian varieties}, Aspects Math., E33.
  Friedr. Vieweg, Braunschweig, 1999, pp.~109--129.

\bibitem{faber3}
{\sc Faber, C.}
\newblock {New developments regarding the tautological ring of the moduli space
  of curves}.
\newblock In {\em Moduli Spaces in Algebraic Geometry\/} ({2010}), {Abramovich,
  D.}, {Farkas, G.}, and {Kebekus, S.}, Eds., {Report No. 02/2010},
  {Mathematisches Forschungsinstitut Oberwolfach}, pp.~{7--9}.

\bibitem{faberpandharipande}
{\sc Faber, C., and Pandharipande, R.}
\newblock Relative maps and tautological classes.
\newblock {\em J. Eur. Math. Soc. (JEMS) 7}, 1 (2005), 13--49.

\bibitem{getzlerpandharipande}
{\sc Getzler, E., and Pandharipande, R.}
\newblock Virasoro constraints and the {C}hern classes of the {H}odge bundle.
\newblock {\em Nuclear Phys. B 530}, 3 (1998), 701--714.

\bibitem{givental1}
{\sc Givental, A.~B.}
\newblock Gromov-{W}itten invariants and quantization of quadratic
  {H}amiltonians.
\newblock vol.~1. 2001, pp.~551--568, 645.
\newblock Dedicated to the memory of I. G. Petrovskii on the occasion of his
  100th anniversary.

\bibitem{givental2}
{\sc Givental, A.~B.}
\newblock Semisimple {F}robenius structures at higher genus.
\newblock {\em Internat. Math. Res. Notices}, 23 (2001), 1265--1286.

\bibitem{harer}
{\sc Harer, J.}
\newblock Improved stability for the homology of the mapping class groups of
  surfaces.
\newblock Preprint, 1993.

\bibitem{harrismumford}
{\sc Harris, J., and Mumford, D.}
\newblock On the {K}odaira dimension of the moduli space of curves.
\newblock {\em Invent. Math. 67}, 1 (1982), 23--88.
\newblock With an appendix by William Fulton.

\bibitem{ionel}
{\sc Ionel, E.-N.}
\newblock Relations in the tautological ring of {$\mathcal{M}_g$}.
\newblock {\em Duke Math. J. 129}, 1 (2005), 157--186.

\bibitem{janda}
{\sc Janda, F.}
\newblock Tautological relations in moduli spaces of weighted pointed curves.
\newblock arXiv:1306.6580, 2013.

\bibitem{janda2}
{\sc Janda, F.}
\newblock Relations on {$\overline M_{g,n}$} via equivariant {G}romov-{W}itten
  theory of {$\Bbb P^1$}.
\newblock {\em Algebr. Geom. 4}, 3 (2017), 311--336.

\bibitem{liuxu}
{\sc Liu, K., and Xu, H.}
\newblock A proof of the {F}aber intersection number conjecture.
\newblock {\em J. Differential Geom. 83}, 2 (2009), 313--335.

\bibitem{liuxu2}
{\sc Liu, K., and Xu, H.}
\newblock Computing top intersections in the tautological ring of
  {$\mathcal{M}_g$}.
\newblock {\em Math. Z. 270}, 3-4 (2012), 819--837.

\bibitem{looijenga2}
{\sc Looijenga, E.}
\newblock On the tautological ring of {$\mathcal{M}_g$}.
\newblock {\em Invent. Math. 121}, 2 (1995), 411--419.

\bibitem{madsenweiss}
{\sc Madsen, I., and Weiss, M.}
\newblock The stable moduli space of {R}iemann surfaces: {M}umford's
  conjecture.
\newblock {\em Ann. of Math. (2) 165}, 3 (2007), 843--941.

\bibitem{morita2}
{\sc Morita, S.}
\newblock Generators for the tautological algebra of the moduli space of
  curves.
\newblock {\em Topology 42}, 4 (2003), 787--819.

\bibitem{morita}
{\sc Morita, S.}
\newblock Tautological algebras of moduli spaces - survey and prospect.
\newblock {P}reprint,
  http://www.kurims.kyoto-u.ac.jp/~kyodo/kokyuroku/contents/pdf/1991-12.pdf,
  2015.

\bibitem{mumford}
{\sc Mumford, D.}
\newblock Towards an enumerative geometry of the moduli space of curves.
\newblock In {\em Arithmetic and geometry, {V}ol. {II}}, vol.~36 of {\em Progr.
  Math.} Birkh\"auser Boston, Boston, MA, 1983, pp.~271--328.

\bibitem{pandharipandepixton}
{\sc Pandharipande, R., and Pixton, A.}
\newblock Relations in the tautological ring.
\newblock arXiv:1101.2236, 2011.

\bibitem{pandharipandepixton2}
{\sc Pandharipande, R., and Pixton, A.}
\newblock Relations in the tautological ring of the moduli space of curves.
\newblock arXiv:1301.4561, 2013.

\bibitem{pandharipandepixtonzvonkine}
{\sc Pandharipande, R., Pixton, A., and Zvonkine, D.}
\newblock Relations on {$\overline{\mathcal{M}}_{g,n}$} via {$3$}-spin
  structures.
\newblock {\em J. Amer. Math. Soc. 28}, 1 (2015), 279--309.

\bibitem{petersen}
{\sc Petersen, D.}
\newblock Tautological rings of spaces of pointed genus two curves of compact
  type.
\newblock {\em Compos. Math. 152}, 7 (2016), 1398--1420.

\bibitem{petersentommasi}
{\sc Petersen, D., and Tommasi, O.}
\newblock The {G}orenstein conjecture fails for the tautological ring of
  {$\overline{\mathcal{M}}_{2,n}$}.
\newblock {\em Invent. Math. 196}, 1 (2014), 139--161.

\bibitem{pixton}
{\sc Pixton, A.}
\newblock Conjectural relations in the tautological ring of
  $\overline{M}_{g,n}$.
\newblock arXiv:1207.1918, 2012.

\bibitem{teleman}
{\sc Teleman, C.}
\newblock The structure of 2{D} semi-simple field theories.
\newblock {\em Invent. Math. 188}, 3 (2012), 525--588.

\bibitem{yin2}
{\sc Yin, Q.}
\newblock On the tautological rings of {$\mathcal{M}_{g, 1}$} and its universal
  {Jacobian}.
\newblock arXiv:1206.3783, 2012.

\bibitem{yinthesis}
{\sc Yin, Q.}
\newblock {\em Tautological Cycles on Curves and Jacobians}.
\newblock Phd thesis, Radboud Universiteit Nijmegen, 2013.

\bibitem{yin}
{\sc Yin, Q.}
\newblock Cycles on curves and {J}acobians: a tale of two tautological rings.
\newblock {\em Algebr. Geom. 3}, 2 (2016), 179--210.

\end{thebibliography}
\end{document}